\documentclass[reqno,a4paper]{amsart}
\usepackage[utf8]{inputenc}
\usepackage[T1]{fontenc}
\usepackage{lmodern}
\usepackage[english]{babel}
\usepackage{amssymb}
\usepackage[foot]{amsaddr}

\usepackage{graphicx}
\usepackage{wrapfig}
\usepackage{subfig}
\usepackage{enumerate}
\usepackage[usenames,dvipsnames]{color}
\usepackage{dsfont}
\usepackage{microtype}

\usepackage{etoolbox}

\usepackage{mathtools}
\mathtoolsset{showonlyrefs,showmanualtags}

\usepackage{pstricks,pst-plot,pst-math} 
\usepackage{pstricks-add}

\newtheorem{thm}{Theorem}

\newtheorem{lem}[thm]{Lemma}

\theoremstyle{definition}
\newtheorem{defn}[thm]{Definition}

\theoremstyle{remark}
\newtheorem{rem}[thm]{Remark}

\newcommand{\DeclareAutoPairedDelimiter}[3]{%
\expandafter\DeclarePairedDelimiter\csname Auto\string#1\endcsname{#2}{#3}%
\begingroup\edef\x{\endgroup
\noexpand\DeclareRobustCommand{\noexpand#1}{%
\expandafter\noexpand\csname Auto\string#1\endcsname*}}%
\x}

\DeclareAutoPairedDelimiter{\abs}{\lvert}{\rvert}
\DeclareAutoPairedDelimiter{\norm}{\lVert}{\rVert}
\DeclareAutoPairedDelimiter{\bra}{(}{ )}
\DeclareAutoPairedDelimiter{\pra}{[}{]}
\DeclareAutoPairedDelimiter{\set}{\{}{\}}
\DeclareAutoPairedDelimiter{\skp}{\langle}{\rangle}

\DeclareMathAlphabet{\mathup}{OT1}{\familydefault}{m}{n}

\newcommand{\dH}[1]{\mathop{}\!\mathup{d} \mathcal{H}^{#1}}

\let\div\relax
\DeclareMathOperator{\div}{div}

\newcommand{\N}{\mathds{N}}
\newcommand{\R}{\mathds{R}}
\renewcommand{\H}{\ensuremath{\mathcal{H}}}

\newcommand{\cC}{\ensuremath{\mathcal C}}

\newcommand{\cG}{\ensuremath{\mathcal G}}
\newcommand{\cH}{\ensuremath{\mathcal H}}

\newcommand{\cK}{\ensuremath{\mathcal K}}

\newcommand{\cP}{\ensuremath{\mathcal P}}

\usepackage[colorinlistoftodos,prependcaption,textsize=tiny]{todonotes}

\newtoggle{Nebenrechnungen}
\togglefalse{Nebenrechnungen}

\usepackage[pdfusetitle]{hyperref}
\definecolor{darkblue}{rgb}{0,0,0.6}
\hypersetup{
    colorlinks=true,       
    linkcolor=darkblue,          
    citecolor=darkblue,        
    filecolor=darkblue,      
    urlcolor=darkblue           
}

\title[Compact coincidence sets in the fractional obstacle problem]{Characterizing compact coincidence sets in the thin obstacle problem and the obstacle problem for the fractional Laplacian}

\author{Simon Eberle$^1$}
\address{$^1$Faculty of Mathematics, University of Duisburg-Essen, Germany}
\email{simon.eberle@uni-due.de}
\author{Xavier Ros-Oton$^2$}
\address{$^2$Institute of Mathematics, University of Zurich, Switzerland}
\email{xavier.ros-oton@math.uzh.ch}
\author{Georg S. Weiss$^1$}
\email{georg.weiss@uni-due.de}

\let\rho\varrho
\let\phi\varphi
\let\epsilon\varepsilon
\begin{document}

\maketitle

\begin{abstract}
In this paper we give a full classification of global solutions of the obstacle problem for the fractional Laplacian (including the thin obstacle problem) with compact coincidence set and at most polynomial growth in dimension $N \geq 3$. We do this in terms of a bijection onto a set of polynomials describing the asymptotics of the solution. Furthermore we prove that coincidence sets of global solutions that are compact are also convex if the solution has at most quadratic growth. 
\end{abstract}

\section{Introduction}

The characterization of global solutions is an important tool in the analysis of free boundaries and singularities. In the classical obstacle problem this problem has been studied extensively for many decades. There the first results ---almost 90 years old--- originated in potential theory, characterizing null quadrature domains. In 1931 P. Dives showed for $N=3$ in \cite{dive} that compact coincidence sets of global solutions of the classical obstacle problem are ellipsoids.  This fact was reproved by H. Lewy in \cite{Lewy79} in 1979. In 1981 M. Sakai gave a full classification of global solutions of the classical obstacle problem in $N=2$ using complex analysis \cite{Sakai}. Only a few years later E. DiBenedetto-A. Friedman \cite{DiBenedettoFriedman} and A. Friedman-M. Sakai \cite{FriedmanSakai} proved that in each dimension $N \geq 3$, bounded coincidence sets of global solutions of the classical obstacle problem are ellipsoids. In a recent short note two of the authors gave a short proof of this known fact \cite{ellipsoid}.
As to unbounded coincidence sets a complete characterization is still an unsolved problem
(cf. \cite[conjecture on p. 10]{Shahgholian92_conjecture} and
\cite[Conjecture 4.5]{KarpMargulis_bounded_sources}).
However, two of the authors (in collaboration with Henrik Shahgholian) achieved a first partial result towards the classification of global solutions of the classical obstacle problem with \emph{unbounded} coincidence set in dimensions $N \geq 6$ \cite{esw}.
\\
To the best knowledge of the authors a classification of global solutions in the obstacle problem for the fractional Laplacian, including the thin obstacle problem, is completely open.
\\
A major difference between global solutions of the classical obstacle problem and global solutions of the  obstacle problem for the fractional Laplacian is that in the classical obstacle problem non-trivial global solutions have quadratic growth towards infinity, while in the obstacle problem for the fractional operator no a-priori estimate on the growth is available. Another difficulty in the obstacle problem for the fractional Laplacian compared to the classical obstacle problem is that the relation between global solutions and potentials of the respective coincidence set is more involved.
Those two differences to the classical obstacle problem are the reason 
for the more abstract form of the following 
characterization of global solutions.

By an observation by L. Caffarelli and S. Silvestre 
it is sufficient to study a weighted local problem instead of the 
obstacle problem for the fractional Laplacian:
for $s \in (0,1)$ the obstacle problem for the fractional $s$-Laplacian with rapidly decreasing obstacle $\phi \in C^2(\R^N)$ given by
\begin{alignat}{3}
	w &\geq \phi \quad \text{ in } \R^N, \\
	(-\Delta)^s w &= 0 \quad \text{ in } \set {w>\phi} \\
	(-\Delta)^s w &\geq 0 \quad  \text{ in } \R^N
\end{alignat}
is in the case of $\Delta \phi \equiv \text{ const}$\footnote{Note that his is the common assumption also for the obstacle problem for the Laplacian cf. \cite[Definition (Normalized Solutions) p. 387]{Caffarelli-revisited}} equivalent (cf. \cite{CaffarelliSilvestre_Extension_problem} or \cite[derivation of (2.2)-(2.5)]{CaffarelliSalsaSilvestre_Regularity_estimates_fractional_Laplacian}) to the higher dimensional \emph{local} problem with density
\begin{align}\tag*{$(*)_a$} \label{eq:equivalent_local_problem}
\left .
\begin{alignedat}{3}
	u &\geq 0 \quad &&\text{ in } \R^N \times \set {0}, \\
	u(x',x_{N+1}) &= u(x',-x_{N+1}) \quad &&\text{ for all } (x',x_{N+1}) \in \R^N \times \R, \\
	\div \bra { \abs{x_{N+1}}^a ~\nabla u   } &= 0 \quad &&\text{ in } \R^{N+1} \setminus \set {x_{N+1}=0, u=0}, \\ 
		\div \bra { \abs{x_{N+1}}^a ~\nabla u   } &\leq 0 \quad &&\text{ in } \R^{N+1} 
\end{alignedat}
~ \right \}
\end{align}
(which is to be understood in the sense of distributions). Here $a = 1-2s$. In the following we will work with the local formulation of the problem.
 In order to abbreviate notation we will from now on set $L_a u :=  	\div \bra { \abs{x_{N+1}}^a~ \nabla u   } $ and $\cC := \set {u=0} \cap \set {x_{N+1}=0}$.

In this equivalent local formulation we are going to characterize
\begin{align}
	\cG_c := \set { u \in C(\R^{N+1}) : u \text{ satisfies  \ref{eq:equivalent_local_problem}  and Definition  \ref{def:at_most_polynomial_growth} and has bounded coincidence set } \cC    }
\end{align}
 and we will do this in terms of the set of polynomials
\begin{align}
	\cP_0' &:= \left \{p \in \cP(\R^{N+1}) : L_a p \equiv 0 \text{ in } \R^{N+1} ~,~ p(x',0) > 0 \text{ for sufficiently large } \abs{x'},   \right . \\
	&\qquad \left . p \text{ symmetric w.r.t. }  \R^N \times \{0\}   \right \}.
\end{align}
With these two definitions at hand, our main result reads as follows.

\begin{thm} \label{thm:characterization of_global_solutions_of_thin_obst_prob_with_bounded_coincidence_set_and_at_most_poly_growth}
\mbox{}\\
Let $N\ge 3$. Moreover, let $u \in \cG_c$ be any solution of \ref{eq:equivalent_local_problem} with bounded coincidence set and at most polynomial growth. Then, there exists a unique polynomial $p \in \cP_0'$, such that 
\begin{align}\label{eq:definition_of_S}
u(x) = p(x) + v_p(x),
\end{align}
where $v_p \to 0$ as $|x|\to\infty$, and $v_p$ is the unique solution of
\begin{align}\label{eq:thin_obstacle_problem_zero_boundary_values}
\left .
\begin{alignedat}{3}
v_p &\geq -p \quad &&\text{ in } \R^N \times \set {0}, \\
	v(x',x_{N+1}) &= v(x',-x_{N+1}) \quad &&\text{ for all } (x',x_{N+1}) \in \R^N \times \R, \\
L_a v_p &=0 \quad &&\text{ in } \R^{N+1} \setminus \set {x_{N+1}=0~,~ v_p=-p }, \\
L_a v_p &\leq 0 \quad &&\text{ in } \R^{N+1}, \\
v_p(x) &\to 0 \quad &&\text{ uniformly as } \abs{x} \to \infty.
\end{alignedat}
~ \right \}
\end{align}
Conversely, for any polynomial $p \in \cP_0'$, we have that \eqref{eq:definition_of_S} defines a solution of \ref{eq:equivalent_local_problem} with bounded coincidence set.
\\
More precisely, the map $S: \cP_0' \to \cG_c$,  defined by $S(p):=p+v_p$, is a bijection.
\end{thm}

\begin{rem}
Note that the asymptotics of a solution $u \in \cG_c$ is given by $p= S^{-1}(u)$. This means that $S$ is also a bijection between solutions having at most/exactly growth of order $m$ and polynomials of degree $m$ / homogeneous polynomials of degree $m$.
Therefore in the case of solutions with quadratic growth we have proved a more abstract (and weaker) version of the bijection constructed in \cite[(5.4) therein]{DiBenedettoFriedman}. We conjecture that coincidence sets need in general not be ellipsoids.
\end{rem}

Furthermore ---as in the classical obstacle problem--- global solutions of \ref{eq:equivalent_local_problem} with quadratic growth and bounded coincidence set are convex in the plane $\set {x_{N+1}=0}$, and their coincidence sets are convex.

\begin{thm}[Solutions with quadratic growth have convex coincidence set] \label{thm:convexity_concavity_coincidence_set}
\mbox{} \\
Let $u \in \cG_c$ be a global solution of the equivalent local problem \ref{eq:equivalent_local_problem} with bounded coincidence set and with quadratic growth in the sense that there is $C<+\infty$ such that
\begin{align} \label{eq:quadratic_growth}
 \abs{u(x)} \leq C(\abs{x}^2+1) \quad \text{ for all } x \in \R^{N+1}.
\end{align} 
Then
\begin{align}
\partial_{ee} u &\geq  0 \quad \text{ in } \R^{N+1}\text{ for all } e \in \set {x_{N+1} =0} \cap \partial B_1,\text{ and } \\
\partial_{N+1} \bra { \abs{x_{N+1}}^a ~\partial_{N+1}u   } &\leq 0 \quad \text{ in } \R^{N+1}
\end{align}
Moreover, the coincidence set $\set {u=0} \cap \set {x_{N+1}=0}$ is convex.
\end{thm}

\begin{defn}[At most polynomial growth] \label{def:at_most_polynomial_growth}
		\mbox{}\\
		We say that a global solution $u$ of the equivalent local problem \ref{eq:equivalent_local_problem} has \emph{at most  polynomial growth} if there is $m \in \N$ and $C>0$ such that
		\begin{align}
		\abs{u(x)} \leq C (1+\abs{x}^m) \quad \text{ for all } x \in \R^{N+1}.
		\end{align}
\end{defn}

\section{Notation}
Throughout this work $\R^N$ will be equipped with the Euclidean inner product $x \cdot y$ and the induced norm $\abs{x}$. Due to the nature of the problem we will often write $x \in \R^{N+1}$ as $x= (x',x_{N+1}) \in \R^{N} \times \R$. The set $B_r(x)$ will be the open $(N+1)$-dimensional ball of center $x$ and radius $r$. Whenever the center is omitted it is assumed to be $0$.
\\
The measure $\cH^{N}$ denotes the $N$-dimensional Hausdorff measure.
By $\cP(\R^N)$ we mean the set of all (real) polynomials in $\R^N$. 

\section{Main part}

Let us now turn to the proof of our two main theorems.

\begin{proof}[Proof of Theorem \ref{thm:characterization of_global_solutions_of_thin_obst_prob_with_bounded_coincidence_set_and_at_most_poly_growth}] \mbox{}\\
\textbf{Step 1.} \emph{The map $S$ is well defined.}\\
For each $p \in \cP_0'$, $S(p)$ is a solution of \ref{eq:equivalent_local_problem}. Thus all we need to check is existence and uniqueness of $v_p$ and compactness of the coincidence set of $v_p$ which will imply compactness of the coincidence set of $S(p)$.\\
\underline{Existence of $v_p$:} We construct a solution of \eqref{eq:thin_obstacle_problem_zero_boundary_values} using Perron's method as the (pointwise) infimum of all $v \in \cK_p$, where
\begin{align}
	\cK_p := \left \{v \in C(\R^{N+1}) : v \text{ satisfies in the sense of distributions } L_a v \leq 0 \text{ in } \R^{N+1}, \right .\\
	\left . v(x',x_{N+1}) = v(x',-x_{N+1}) ,  v \geq -p \text{ in } \R^N \times \set {0}, v(x) \to 0 \text{ as } \abs{x} \to \infty  \right \}.
\end{align}
Note that $\cK_p$ is nonempty since for sufficiently large $c>0$ it contains the function
\begin{align}
	w_c(x) :=  \int \limits_{\R^{N} \times \set {0} \cap \set {-p>0} } \frac{c}{ \abs{x-y}^{N-1+a}   } \dH{N}(y). 
\end{align}
Then by the usual arguments in Perron's method we find that
\begin{align}
	v_p := \inf \limits_{v \in \cK_p} v
\end{align}
satisfies 
\begin{align}
v_p &\geq -p ~\text{ in } \R^N \times \set {0} ~,~
L_a v_p =0 ~ \text{ in } \R^{N+1} \setminus \set {x_{N+1}=0, v_p=-p } ~,~	L_a v_p \leq 0 ~ \text{ in } \R^{N+1} \quad 
\end{align}
and the fact that $w_c \in \cK_p$ combined with the maximum principle for the operator $L_a$ (cf. \cite[Theorem 2.2.2]{FabesKenigSerapioni}) we obtain
\begin{align}
	0 \leq v_p(x) \leq w_c(x) \to 0 \quad \text{ as } \abs{x} \to \infty.
\end{align}
 \\
\underline{Uniqueness of $v_p$:} Let us assume that there is another solution $v$ of \eqref{eq:thin_obstacle_problem_zero_boundary_values}. By construction of the Perron solution $v_p$ clearly $v_p \leq v$ in $\R^{N+1}$. Furthermore for every $\epsilon >0$ there is $R_0(\epsilon) >0$ such that for all $R > R_0(\epsilon)$
\begin{align}
v \leq \epsilon \quad \text{ in } \R^{N+1} \setminus B_R 
\end{align} 
and therefore also
\begin{align}
v \leq \epsilon + v_p \quad \text{ on } \partial B_R.
\end{align}
Applying the (local) comparison principle Lemma \ref{lem:comparsion_principle_for_the_nonlinear_equation} with $u_1 = v$ and $u_2 = v_p + \epsilon$ and $\phi = -p$ in $U =B_R$ implies that $v \leq v_p +\epsilon$ in $B_R$ and hence (by choice of $R$) in $\R^{N+1}$. Putting it all together we have for all $\epsilon >0$ that
\begin{align}
	v_p \leq v \leq v_p + \epsilon \quad \text{ in } \R^{N+1}
\end{align}
and therefore $v_p \equiv v$.\\
\underline{Compactness of the coincidence set of $v_p$:}
\begin{itemize}
\item either $-p  \leq  0$ on $\R^N \times \set{0}$ in which case we infer from the uniqueness of $v_p$ that $v_p \equiv 0$ in $\R^{N+1}$, or
\item $\set {-p >0} \cap \bra {\R^N \times \set {0}}$ is bounded due to the asymptotic behavior of $p$ and (relatively) open. In this case we know furthermore that $\{-p >0 \}$ is nonempty. Then $p$ ---being a polynomial--- satisfies $p(x',0) \to \infty$ as $\abs{x'} \to +\infty$. (This follows from the simple fact that a polynomial can only be bounded (from above and below) in any direction if it is already constant in that direction. But $p$ being constant in any direction together with the fact that $\{-p>0\}$ is open and nonempty contradicts the asymptotic behavior of $p$ in $\R^N \times \{0\}$.) Combining the asymptotics of $v_p$ and $p$ in $\R^N \times \set {0}$ we obtain that $\set {v_p =-p} \cap \bra { \R^N \times \set {0}   }$ is bounded and hence compact.
\end{itemize}
\textbf{Step 2.} \emph{The map $S$ is injective.}\\
Suppose that there are $p, \tilde{p} \in \cP_0'$ such that $p \neq \tilde{p}$. Then $S(p) \to p$ and $S(\tilde{p}) \to \tilde{p}$ as $\abs{x} \to \infty$ and thus $S(p) \neq S(\tilde{p})$.\\
\textbf{Step 3.} \emph{The map $S$ is surjective.}\\
Let $u \in \cG_c$ be arbitrary.
Then  $u$ solves (in the sense of distributions)
\begin{align}\label{eq:Euler-Lagrange-equation_of_u}
L_a u =2~  \pra{\abs{x_{N+1}}^a \partial_{N+1} u}_+~ \H^{N}\lfloor_{\set {x_{N+1} =0~,~ u=0}} \quad \text{ in } \R^{N+1},
\end{align}
where we mean by $\pra{\abs{x_{N+1}}^a \partial_{N+1} u}_+(x') = \lim \limits_{x_{N+1} \searrow 0} \abs{x_{N+1}}^a \partial_{N+1} u(x',x_{N+1}) $. 
Let us define the `potential-solution' for the right-hand side depending on the solution $u$, i.e.
\begin{align} \label{eq:representation_of_v_p}
v(x) := \alpha_{N+1+a} \int \limits_{\R^{N+1}} \frac{-2 \pra{\abs{y_{N+1}}^a \partial_{N+1} u}_+(y') }{\abs {x-y}^{N-1+a}}  \dH{N}\lfloor_{\set {y_{N+1} =0~,~u=0}}(y),
\end{align}
where $\alpha_{N+1+a}>0$ is such that $v$ solves (in the sense of distributions)
\begin{alignat}{3}
L_a v &=2~  \pra{\abs{x_{N+1}}^a \partial_{N+1} u}_+~ \H^{N}\lfloor_{\set {x_{N+1} =0~,~ u=0}}  &&~\text{ in } \R^{N+1}  ,  \label{eq:PDE_of_v} \\
 v(x',x_{N+1}) &= v(x', -x_{N+1}) && ~\text{ for all } (x',x_{N+1}) \in \R^N \times \R.
\end{alignat}
This fact can be found either in \cite[Section 2.2]{CaffarelliSilvestre_Extension_problem} or obtained by direct calculation.
From the assumption that $\set {u=0} \cap \set {x_{N+1} =0}$ is compact we conclude from the regularity theory for solutions of \ref{eq:equivalent_local_problem} (cf. \cite[Lemma 4.1]{CaffarelliSalsaSilvestre_Regularity_estimates_fractional_Laplacian}) that $\pra{\abs{x_{N+1}}^a \partial_{N+1} u}_+ \lfloor_{\set {u=0~,~x_{N+1} =0}}$ is bounded. This implies that 
\begin{align} \label{eq:limit_of_v}
\abs{v(x) }\to 0 \quad \text{ uniformly as } \abs{x} \to \infty. 
\end{align}
Combining \eqref{eq:Euler-Lagrange-equation_of_u},\eqref{eq:PDE_of_v} and \eqref{eq:limit_of_v} and the assumption that $u$ has at most polynomial growth we conclude that
\begin{align}
L_a (u-v) \equiv 0 \quad \text{ in } \R^{N+1},
\end{align}
and invoking the Liouville type theorem \cite[Lemma 2.7]{CaffarelliSalsaSilvestre_Regularity_estimates_fractional_Laplacian}  there is a polynomial $p$ such that 
\begin{align}
u-v \equiv p ~ \text{ , } ~ L_a p \equiv 0 \quad  \text{ in } \R^{N+1}
\end{align}
and the function $v =u-p$ solves 
\begin{align}
\begin{cases}
v \geq -p \quad &\text{ in } \R^{N} \times \set {0}, \\
L_a v = 0 \quad &\text{ in } \R^{N+1} \setminus \set {x_{N+1} =0 ~,~ v = -p}, \\
L_a v \leq 0 \quad &\text{ in } \R^{N+1}, \\
v(x) \to 0 \quad &\text{ uniformly as } \abs{x} \to \infty.
\end{cases}
\end{align}
But this means exactly that $v$ is a solution of \eqref{eq:thin_obstacle_problem_zero_boundary_values}. Since $u \in \cG_c$ has a compact coincidence set and $\set {u=0} \cap \bra {\R^N \times \set {0}} = \set {v=-p} \cap \bra {\R^N \times \set {0}} $ we infer that $v$, too, has a bounded coincidence set. And since $p$ is a polynomial this implies that   $p(x',0) > 0 \text{ for sufficiently large } \abs{x'}$.   
\\
(This can be seen as follows. Either $p(x',0)>0$ for all $x' \in \R^N$. Then by uniqueness of solutions of \eqref{eq:thin_obstacle_problem_zero_boundary_values} (cf. Step 1) $v \equiv 0$ in $\R^{N+1}$. Together with the fact that $v$ has compact coincidence set $\{v= -p\} \cap (\R^N \times \{0\})$ this implies that $p(x',0)>0$ for sufficiently large $|x'|$. On the other hand if there is $x_0' \in \R^N$ such that $p(x_0',0)<0$ we conclude from \eqref{eq:thin_obstacle_problem_zero_boundary_values} that for every $\epsilon >0$ there is $R(\epsilon)>0$ such that  $p(x',0) \geq - \epsilon$ for all $x' \in \R^N \setminus B_{R(\epsilon)}$. Since $p$ is a polynomial this implies that $p(x',0) \geq 0$ for sufficiently large $x'$. let us now assume towards a contradiction that there is $({x'}^n)_{n \in \N} \subset \R^N$, $|{x'}^n| \to \infty$ as $n \to \infty$ such that $p({x'}^n,0) =0$ for all $n \in \N$. Then, being a polynomial, $p$ must be constantly equal to zero in some direction $e' \in \partial B_1(x_0',0) \cap \{x_{N+1 } =0\}$. But this contradicts the assumption that $p(x_0',0)<0$.)
\end{proof}

\begin{proof}[Proof of Theorem \ref{thm:convexity_concavity_coincidence_set}]
Let us set for each $h \in \set {x_{N+1}=0}$
\begin{align}
u_h(x) := u(x+h) -2u(x) +u(x-h) \quad \text{for } x \in \R^{N+1}.
\end{align}From $u$ being a solution of \ref{eq:equivalent_local_problem} we infer that
\begin{align} \label{eq:u_h_is_superharmonic}
L_a u_h \leq 0 \quad \text{ in } \R^{N+1} \setminus \bra { \set {x_{N+1}=0} \cap \set {u=0}}.
\end{align}
Theorem \ref{thm:characterization of_global_solutions_of_thin_obst_prob_with_bounded_coincidence_set_and_at_most_poly_growth} implies that the asymptotics of $u$ as $\abs {x} \to \infty$ is given by $p:= S^{-1}(u) \in \cP_0'$, and combining this with the quadratic growth of $u$ we obtain that
either we are in the trivial case $p(x',0) \geq 0$ for all $x' \in \R^N$ whence by uniqueness of $v_p$ we obtain that $u \equiv p$ and $p(x',0) >0$ for large $|x'|$ combined with the quadratic growth of $p$ implies that $\partial _{ee} p \geq 0$ in $\R^{N+1}$ for all $e \in \partial B_1 \cap \{ x_{N+1} =0 \}$ and the claim of the theorem follows.\\
If there is $x_0' \in \R^N$ such that $p(x_0',0) <0$ then $p \in \cP_0'$ and $p$ being quadratic implies that 
\begin{align}\label{eq:convexity_of_p_in_x_N+1=0_directions}
\partial_{ee} p(x) >0 \quad \text{ for each } e \in \set {x_{N+1}=0} \cap \partial B_1. 
\end{align}
Otherwise there would be $\tilde{e} \in \partial B_1 \cap \{x_{N+1} =0\}$ such that $\partial_{\tilde{e} \tilde{e}} \leq 0$ in $\R^{N+1}$. But this combined with the assumption that $p(x_0',0)<0$ contradicts the fact that $p \in \cP_0'$.\\
For sufficiently large $R$ we have 
 using the integral representation of $v=u-p$ in \eqref{eq:representation_of_v_p} that
\begin{align}
	\abs{\partial_{ij} v(x)} \leq C(R)~ \abs{x}^{-(N+1+a)}  \quad \text{ for all } x  \in \R^{N+1} \setminus B_R,~ i,j \in \set {1, \dots, N+1}.
\end{align}
Combining this fact with  \eqref{eq:convexity_of_p_in_x_N+1=0_directions} we obtain that there is $R>0$ such that 
\begin{align}
\partial_{ee} u(x) \geq0 \quad \text{ for all } e \in \set {x_{N+1}=0} \cap \partial B_1 \text{ and } \abs{x}>R. 
\end{align}
This implies that for each $h \in \set {x_{N+1}=0}$ there is $R(h)>0$ such that 
\begin{align} \label{eq:u_h_has_sign_far_away}
u_h(x) \geq 0 \quad \text{ for all } \abs{x} \geq R(h).
\end{align} 
Furthermore, recalling the definition of $u_h$ and the fact that $u \geq 0$ in $\R^N \times \{0\}$ we have that
\begin{align} \label{eq:u_h_has_sign_on_coincidence_set}
u_h(x) \geq 0 \quad \text{for all } x \in \set {u=0} \cap \set {x_{N+1}=0}.
\end{align}
Combining \eqref{eq:u_h_is_superharmonic}, \eqref{eq:u_h_has_sign_far_away} and \eqref{eq:u_h_has_sign_on_coincidence_set} and the minimum principle for supersolutions of the operator $L_a$ (cf. \cite[Theorem 2.2.2]{FabesKenigSerapioni}) we obtain that
\begin{align}
u_h(x) \geq 0 \quad \text{ for all } \abs{x} \leq R(h).
\end{align}
Combining this with \eqref{eq:u_h_has_sign_far_away} we conclude that
\begin{align}
u_h \geq 0 \quad \text{ in } \R^{N+1}.
\end{align}
This implies 
 that $u$ is convex in any direction $e \in \set {x_{N+1}=0} \cap \partial B_1$ and furthermore that the coincidence set $\set {u=0} \cap \set {x_{N+1}=0}$ is convex. 

Combining the above convexity of $u$ in the plane $\set {x_{N+1}=0}$ with $L_a u \leq 0$ in $\R^{N+1}$ (cf. \ref{eq:equivalent_local_problem}), we conclude that
\begin{align}
\partial_{N+1} \bra { \abs{x_{N+1}}^a \partial_{N+1} u } \leq 0 \quad \text{ in } \R^{N+1}.
\end{align}
\end{proof}


\appendix
\section*{Appendix}
Since we could not find a reference in the literature we have included for the sake of completeness the following comparison lemma.
\begin{lem}[Local comparison principle for solutions of \ref{eq:equivalent_local_problem}] \label{lem:comparsion_principle_for_the_nonlinear_equation}
	\mbox{}\\
	Let $U \subset \R^{N+1}$ be a bounded domain, symmetric to $\set {x_{N+1}=0}$, $\varphi \in C^2(\bar{U})$ be symmetric and $u_1, u_2 \in C(\bar{U})$ be two (distributional) solutions of 
\begin{alignat}{3}\label{eq:localalzed_fractional_problem}
u_i &\geq \varphi \quad &&\text{ in } U \cap \set {x_{N+1}=0}, \\
u_i(x',x_{N+1}) &= u_i(x',-x_{N+1}) \quad &&\text{ for  all } (x',x_{N+1}) \in U, \\
L_a u_i   &\leq 0 \quad &&\text{ in } U, \\
L_a u_1  &=0 \quad &&\text{ in } U \setminus \{ x_{N+1}=0, u_1=\varphi  \}, 
\end{alignat}
for $i \in \set {1,2}$ such that 
\begin{align}
u_1 \leq u_2 \quad \text{ on } \partial U,
\end{align}	
then
\begin{align}
u_1 \leq u_2 \quad \text{ in } U.
\end{align}
\end{lem}

\begin{proof}
	Let $\tilde{U} := U \setminus \bra { \set {u_1 = \varphi} \cap \set {x_{N+1}=0 }}$. Then it holds that
	\begin{alignat}{3}
		L_a u_1&= 0 \quad &&\text{ in } \tilde{U},\\
			L_a u_2 &\leq 0 \quad &&\text{ in } \tilde{U},\\
			u_1 &\leq u_2 \quad &&\text{ on } \partial \tilde{U}
	\end{alignat}
	and therefore invoking the maximum principle for the operator $L_a$ (cf. \cite[Theorem 2.2.2]{FabesKenigSerapioni}) the function $u_1-u_2$ attains its non-positive maximum at the boundary $\partial \tilde{U}$ and hence
	\begin{align}
	u_1-u_2 \leq 0\quad \text{ in } U.
	\end{align}
\end{proof}

\bibliographystyle{abbrv}
\bibliography{compact_Signorini_references}

\end{document}